\documentclass[11pt]{amsart}

\usepackage{amsmath,amssymb}
\usepackage{amsthm}
\usepackage{enumerate}
\usepackage{multicol}
\usepackage{hyperref}
\usepackage{url}

\usepackage{color}

\usepackage{listings}
\usepackage{tikz}
\usetikzlibrary{matrix,arrows}

\newcommand{\diag}{\mathop{\mathrm{diag}}}

\newtheorem{theorem}{Theorem}
\newtheorem{lemma}[theorem]{Lemma}
\newtheorem{proposition}[theorem]{Proposition}
\newtheorem{corollary}[theorem]{Corollary}

\theoremstyle{definition}

\newtheorem{remark}[theorem]{Remark}
\newtheorem{claim}[theorem]{Claim}
\newtheorem{example}[theorem]{Example}

\title{Distance Ideals of Graphs}

\begin{document}
\maketitle
\begin{center}
\author{Carlos A. Alfaro}\footnotemark[1]
\author{Libby Taylor}\footnotemark[2]
\end{center}
\footnotetext[1]{
Banco de M\`exico, Mexico City, Mexico (carlos.alfaro@banxico.org.mx)
}
\footnotetext[2]{
School of Mathematics, Georgia Institute of Technology, 686 Cherry St, Atlanta, GA 30332
(libbytaylor@gatech.edu)
}
\footnotetext[3]{
C.A. Alfaro was partially supported by SNI and CONACyT.
}

\begin{abstract}
We introduce the concept of distance ideals of graphs, which can be regarded as a generalization of the Smith normal form and the spectra of the distance matrix and the Laplacian distance matrix of a graph.
We also obtain a classification of the graphs with at most one trivial distance ideal.
\end{abstract}

\section{Introduction}

Let $M$ be an integral matrix, and let $\diag(f_1, f_2, \dots, f_{r})$ be its Smith normal form, so that  $f_1, f_2, \dots, f_{r}$ are positive integers such that $f_i \mid f_j$ for all $i\leq j$.
These integers are called the {\it invariant factors} of $M$.
Computing the Smith normal form of matrices has been of interest in combinatorics.
For instance, computing the Smith normal form of the adjacency or Laplacian matrix is a standard technique used to determine the Smith group and the critical group of a graph; see \cite{alfaval0,merino,rushanan}.
The critical group of a connected graph is especially interesting since, by Kirchoff's matrix-tree theorem, its order is equal to the number of spanning trees of the graph.
The study of the invariant factors of combinatorial matrices seems to have started in \cite{N} and was soon continued in \cite{WW}.
 We refer the reader to \cite{stanley} for a survey on the Smith normal forms in combinatorics for more details in the topic.

Let $G=(V,E)$ be a connected graph.
The {\it distance} $d_G(u,v)$ between the vertices $u$ and $v$ is the number of edges in a shortest path between them.  The {\it distance matrix} $D(G)$ of $G$ is the matrix with rows and columns indexed by the vertices of $G$ with the $uv$-entry equal to $d_G(u,v)$.
Distance matrices were introduced by Graham and Pollack in the study of a data communication problem in \cite{GP}.
This problem involved finding appropriate addresses so that a message can move efficiently through a series of loops from its origin to its destination, choosing the best route at each switching point.


Little is known about the Smith normal forms of distance matrices.
In \cite{HW}, the Smith normal forms of the distance matrices were determined for trees, wheels, cycles, and complements of cycles and were partially determined for complete multipartite graphs.
In \cite{BK}, the Smith normal form of the distance matrices of
unicyclic graphs and of the wheel graph with trees attached to each
vertex were obtained.

It is well known that the Smith normal form of a matrix over a principal ideal domain ({\it p.i.d.}) can be computed using row and column operations.
In fact, in \cite{KB}, Kannan and Bachem found polynomial algorithms for computing the Smith normal form of an integer matrix.
An alternative way of determining the Smith normal form is as follows.
Let $\Delta_i(G)$ denote the {\it greatest common divisor} of the $i$-minors of the distance matrix $D(G)$.  Then the $i$-{\it th} invariant factor $d_i$ is equal to $\Delta_i(G)/ \Delta_{i-1}(G)$, where $\Delta_0(G)=1$.
We will generalize on this method to develop the notion of distance ideals.

The paper is organized as follows.
In Section \ref{section:DefinitionDistanceIdeals}, we define distance ideals and explore their varieties, as well as their behaviour under taking induced subgraphs.
We finish this section by giving a description of the distance ideals of the complete graphs and star graphs.
In Section~\ref{section:classification}, we will give a classification of the graphs which have exactly 1 trivial distance ideal over $\mathbb{Z}$ and $\mathbb{R}$ in terms of forbidden induced subgraphs.


\section{Distance ideals}\label{section:DefinitionDistanceIdeals}
Through the paper, we will assume all graphs are connected.
Given a connected graph $G=(V,E)$ and a set of indeterminates $X_G=\{x_u \, : \, u\in V(G)\}$, let $\diag(X_G)$ denote the diagonal matrix with the indeterminates in the diagonal and zeroes elsewhere.
The {\it generalized distance matrix} $D(G,X_G)$ of $G$ is the matrix with rows and columns indexed by the vertices of $G$ defined as $\diag(X_G)+D(G)$.
Note we can recover the distance matrix from the generalized distance matrix by evaluating $X_G$ at the zero vector, that is, $D(G)=D(G,\bf{0})$.

Let $\mathcal{R}[X_G]$ be the polynomial ring over a commutative ring $\mathcal{R}$ in the variables $X_G$.
For all $i\in[n]:=\{1,..., n\}$, the $i${\it-th distance ideal} $I^\mathcal{R}_i(G,X_G)$ of $G$ is the determinantal ideal given by
\[
\langle {\rm minors}_i(D(G,X_G))\rangle\subseteq \mathcal{R}[X_G],
\]
where $n$ is the number of vertices of $G$ and ${\rm minors}_i(D(G,X_G))$ is the set of the determinants of the $i\times i$ submatrices of $D(G,X_G)$.
Computing the Gr\"obner basis of the distance ideals gives us a compact description of these ideals.

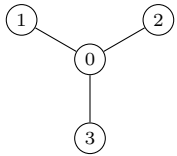
\begin{figure}[h!]
\begin{center}
\begin{tikzpicture}[scale=.7]
 	\tikzstyle{every node}=[minimum width=0pt, inner sep=2pt, circle]
    \draw (0,0) node[draw] (0) {\tiny $0$};
 	\draw (30:1.5) node[draw] (1) {\tiny $2$};
 	\draw (150:1.5) node[draw] (2) {\tiny $1$};
 	\draw (270:1.5) node[draw] (3) {\tiny $3$};
 	\draw (0) -- (1);
 	\draw (0) -- (2);
 	\draw (0) -- (3);
 \end{tikzpicture}
\end{center}
\caption{Claw graph $K_{1,3}$.}
\label{figure:lambda}
\end{figure}

\begin{example}
The generalized distance matrix of the claw graph $K_{1,3}$ is the following.
\[
D(K_{1,3},X_{K_{1,3}})=
\begin{bmatrix}
x_0 & 1 & 1 & 1\\
1 & x_1 & 2 & 2\\
1 & 2 & x_2 & 2\\
1 & 2 & 2 & x_3\\
\end{bmatrix}
\]
For this example, we will consider the distance ideals over $\mathbb{Z}[X_{K_{1,3}}]$.
It is obvious that $I^\mathbb{Z}_1(K_{1,3},X_{K_{1,3}})=\langle 1 \rangle$, since the $(0,1)$-entry of the generalized distance matrix is equal to 1.
The Gr\"obner basis of the second distance ideal $I^\mathbb{Z}_2(K_{1,3},X_{K_{1,3}})$ is
\[
\langle 2x_0 - 1, x_1 - 2, x_2 - 2, x_3 - 2\rangle.
\]
The Gr\"obner basis of $I^\mathbb{Z}_3(K_{1,3},X_{K_{1,3}})$ is equal to
\begin{eqnarray*}
\langle 2x_0x_1 - 4x_0 - x_1 + 2, 2x_0x_2 - 4x_0 - x_2 + 2, 2x_0x_3 - 4x_0 - x_3 + 2, \\
x_1x_2 - 2x_1 - 2x_2 + 4, x_1x_3 - 2x_1 - 2x_3 + 4, x_2x_3 - 2x_2 - 2x_3 + 4\rangle.
\end{eqnarray*}
Finally, the Gr\"obner basis of $I^\mathbb{Z}_4(K_{1,3},X_{K_{1,3}})$ is
\[
\langle x_0x_1x_2x_3 - 4x_0x_1 - 4x_0x_2 - 4x_0x_3 + 16x_0 - x_1x_2 - x_1x_3 + 4x_1 - x_2x_3 + 4x_2 + 4x_3 - 12\rangle.
\]
\end{example}

At the end of this section, we will compute the distance ideals of the star graphs, which is a family of graphs containing the claw.

An ideal is said to be {\it unit} or {\it trivial} if it is equal to $\langle1\rangle$. 
Let $\Phi_\mathcal{R}(G)$ denote the maximum integer $i$ for which $I^\mathcal{R}_i(G,X_G)$ is trivial.
Note that every graph with at least one non-loop edge has at least one trivial distance ideal.

It has been of interest to study graphs whose Smith normal form of its associated matrix (say Laplacian matrix or adjacency matrix) has a particular number of invariant factors equal to 1.
This is because this number is related to the cyclicity of the group obtained from cokernel of the matrix.
Let $\phi_\mathcal{R}(G)$ denote the number of invariant factors over a p.i.d. $\mathcal{R}$ of the distance matrix of $G$ equal to 1.

The following observation will give us the relation between the Smith normal form of the distance matrix and the distance ideals over a p.i.d.
\begin{proposition}\label{teo:eval1}
Let $\mathcal{R}$ be a p.i.d. and ${\bf d}\in \mathcal{R}^{V(G)}$.
If $f_1\mid\cdots\mid f_{r}$ are the invariant factors of the matrix $D(G,{\bf d})$ over $\mathcal{R}$, then
\[
I^{\mathcal{R}}_i(G,{\bf d})=\left\langle \prod_{j=1}^{i} f_j \right\rangle\text{ for all }1\leq i\leq r.
\]
\end{proposition}
\begin{proof}
    Let $\Delta^\mathcal{R}_i$ denote the {\it g.c.d.} over $\mathcal{R}$ of ${\rm minors}_i(D(G,{\bf d}))$.
    We have $\langle {\rm minors}_i(D(G,{\bf d}))\rangle = \langle \Delta^\mathcal{R}_i \rangle$.
    Since $f_i=\Delta^\mathcal{R}_i/\Delta^\mathcal{R}_{i-1}$ with $\Delta^\mathcal{R}_0=1$, then $I^{\mathcal{R}}_i(G,{\bf d})=\left\langle \prod_{j=1}^{i} f_j \right\rangle$.
\end{proof}

In this way, 
to recover the Smith normal form of $D(G)$ from the distance ideals, we just need to evaluate them $X_G$ at ${\bf 0}$.
Moreover, 
if the $i$-th invariant factor, computed over $\mathcal{R}$, of $D(G,{\bf d})$ is not equal to $1$, then the ideal $I^\mathcal{R}_i(D, X_D)$ is not trivial.
Another consequence of Proposition~\ref{teo:eval1} is the following.
\begin{corollary}
    For any graph $G$, $\Phi_\mathcal{R}(G)\leq \phi_\mathcal{R}(G)$.
    In particular, for any positive integer $k$, the family of graphs with $\Phi_\mathcal{R}(G)\leq k$ contains the family of graphs with $\phi_\mathcal{R}(G)\leq k$.
\end{corollary}
\begin{proof}
    The inequality follows by observing that if the distance ideal $I^\mathcal{R}_i(G,X_G)$ is trivial, then $\Delta^\mathcal{R}_i(D(G))=1$, and thus the $i$-th invariant factor is equal to $1$.
    Now, let $G$ be a graph with $\phi_\mathcal{R}(G)\leq k$.
    Then by previous equation, $\Phi_\mathcal{R}(G)\leq\phi_\mathcal{R}(G)\leq k$.
\end{proof}

In Section~\ref{section:classification}, we will give some characterizations of graphs with 1 trivial distance ideals.
Meanwhile, it is not difficult to see that the family of graphs with $\phi_\mathcal{R}(G)\leq 1$ consists only of the graph with one vertex.
In fact, there is no graph with $\phi_\mathcal{R}(G)=1$.

\subsection{Varieties of distance ideals}

Let $I\subseteq \mathcal{R}[X]$ be an ideal in $\mathcal{R}[X]$.
The variety associated to the ideal $I$ is
\[
V_\mathcal{R}(I)=\left\{ {\bf a}\in \mathcal{R}^n : g({\bf a}) = 0 \text{ for all } g\in I \right\}.
\]
Note that if $I$ is trivial, then $V_\mathcal{R}(I)=\emptyset$.

Let $M$ be an $(i+1)\times (i+1)$-matrix with entries in $\mathcal{R}[X_G]$.
We have
\[
\det M = \sum_{j=1}^{i+1}M_{j,1}\det M[j;1],
\]
where $M_{j,1}$ denotes the $(j,1)$ entry of the matrix and $M[j;1]$ denotes the submatrix of $M$ whose $j$-th row and $1$st column were deleted.
More general, $M[\mathcal{I;J}]$ denote the sumbratix of a matrix $M$ generated by eliminating the rows and columns with indices in $\mathcal{I}$ and $\mathcal{J}$, respectively.
For simplicity, when $\mathcal{I=J}$, we just write $M[\mathcal{I}]$.
This gives that $I^\mathcal{R}_{i+1}(G,X_G)\subseteq I^\mathcal{R}_i(G,X_G)$.
Thus, distance ideals satisfy the condition that
\[
\langle 1\rangle \supseteq I^\mathcal{R}_1(G,X_G) \supseteq \cdots \supseteq I^\mathcal{R}_n(G,X_G) \supseteq \langle 0\rangle.
\]
Therefore
\[
\emptyset=V_\mathcal{R}(\langle 1\rangle) \subseteq V_\mathcal{R}(I^\mathcal{R}_1(G,X_G)) \subseteq \cdots \subseteq V_\mathcal{R}(I^\mathcal{R}_n(G,X_G)) \subseteq V_\mathcal{R}(\langle 0\rangle)=\mathcal{R}^n.
\]

If $V_\mathcal{R}(I^\mathcal{R}_k(G,X_G))\neq\emptyset$ for some $k$, then there exists ${\bf a}\in\mathcal{R}^n$ such that, for all $t \geq k$, $I^\mathcal{R}_{t}(G,{\bf a})=\langle 0\rangle$; that is, all $t$-minors of $D(G,{\bf a})$ have determinant equal to 0.
In particular, these varieties can be regarded as a generalization of the distance spectra of $G$.
Distance spectra of graphs have been widely studied; see for example the recent surveys \cite{AH,SI}.
Let $I^\mathcal{R}_i(G,\lambda)$ denote the $i$-{\it th} distance ideal where each $x_i=\lambda$ for all $i\in[n]$.
Therefore, we have that $I^\mathcal{R}_n(G,-\lambda)=\langle \det(-\lambda {\sf I}_n + D(G))\rangle$, and the variety of this ideal is the negative of the distance spectra of $G$.
In particular, if $\lambda$ is a graph eigenvalue of the distance matrix, then $I^\mathcal{R}_n(G,-\lambda)=\langle 0\rangle$.

\begin{example}
For the complete graph $K_3$ with 3 vertices, the Gr\"obner basis of the second distance ideal $I^\mathbb{R}_2(K_3,X_{K_3})$ is equal to $\langle x_0 - 1, x_1 - 1, x_2 - 1\rangle$, and the third distance ideal $I^\mathbb{R}_3(K_3,X_{K_3})$ is equal to $\langle x_0 x_1 x_2 - x_0 - x_1 - x_2 + 2\rangle$.
The variety of $I^\mathbb{R}_2(K_3,X_{K_3})$ consists only of the vector $(1,1,1)$, but the variety of $I^\mathbb{R}_3(K_3,X_{K_3})$ is more interesting; see Figure~\ref{fig:varietyK3}.
By evaluating, we have that $I^\mathbb{R}_2(K_3,-\lambda)=\langle \lambda+1\rangle$, whose variety consists only of $-1$, and the ideal $I^\mathbb{R}_3(K_3,-\lambda)=\langle \lambda^3-3\lambda-2\rangle$ has variety $V_\mathbb{R}(I^\mathbb{R}_3(K_3,\lambda))=\{2,-1\}$.

\begin{figure}[h!]
    \begin{center}
        \includegraphics[scale=0.5]{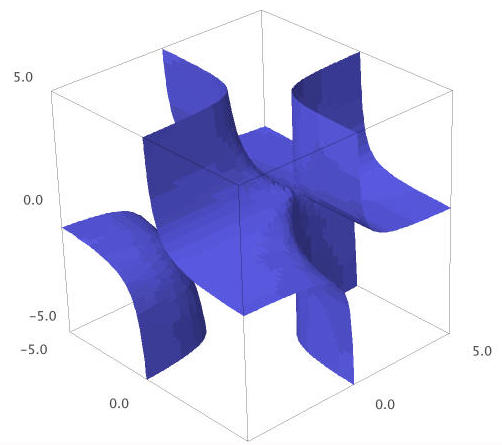}
    \end{center}
    \caption{Partial view of the variety of $I^\mathbb{R}_2(K_3,X_{K_3})$ in $\mathbb{R}^3$.}
    \label{fig:varietyK3}
\end{figure}

\end{example}

%
%

\subsection{Distance ideals of induced subgraphs}

In general, distance ideals are not monotone under taking induced subgraphs.
A counterexample can be constructed, for example, from $P_5$ considered as induced subgraph of $C_6$, since the distance of the leaves of $P_5$ in $C_6$ is 2.
However, we have the following result:

\begin{lemma}\label{lemma:inducemonotone}
Let $H$ be an induced subgraph of $G$ such that 
for every pair of vertices $v_i,v_j$ in $V(H)$, there is a shortest path from $v_i$ to $v_j$ in $G$ which lies entirely in $H$.
Then $I^\mathcal{R}_i(H,X_H)\subseteq I^\mathcal{R}_i(G,X_G)$  for all $i\leq |V(H)|$ and $\Phi_\mathcal{R}(H)\leq \Phi_\mathcal{R}(G)$.
\end{lemma}
\begin{proof}
Since any $i\times i$ submatrix of $D(H,X_H)$ is an $i\times i$ submatrix of $D(G,X_G)$, we have $I^\mathcal{R}_i(H,X_H)\subseteq I^\mathcal{R}_i(G,X_G)$  for all $i\leq |V(H)|$.
Thus if $I^\mathcal{R}_i(H,X_H)$ is trivial for some $i$, then $I^\mathcal{R}_i(G,X_G)$ is trivial.
\end{proof}

In particular we have the following.

\begin{lemma}\label{lemma:distance2inducedmonotone}
Let $H$ be an induced subgraph of $G$ with diameter is 2, that is the distance between any pair of vertices in $H$ is at most 2.
Then $I^\mathcal{R}_i(H,X_H)\subseteq I^\mathcal{R}_i(G,X_G)$  for all $i\leq |V(H)|$.
\end{lemma}

A related family of graphs, defined in \cite{H77}, are distance-hereditary graphs.
A graph $G$ is {\it distance-hereditary} if for each connected induced subgraph $H$ of $G$ and every pair $u$ and $v$ of vertices in $H$, $d_H(u,v)=d_G(u,v)$.

\begin{proposition}
    Let $G$ be a distance hereditary graph.
    If $H$ is a connected induced subgraph of $G$, then $I^\mathcal{R}_i(H,X_H)\subseteq I^\mathcal{R}_i(G,X_G)$  for all $i\leq |V(H)|$, and $\Phi_\mathcal{R}(H)\leq \Phi_\mathcal{R}(G)$.
\end{proposition}

There are other interesting examples not considered in Lemma~\ref{lemma:inducemonotone}.

\begin{example}\label{example:P4inducemonotone}
Let $P_4$ be the path with $V(P_4)=\{v_1,v_2,v_3,v_4\}$ and $E(P_4)=\{v_1v_2, v_2v_3, v_3v_4\}$.  
Let $G$ be a graph containing $P_4$ as induced subgraph.
The only way to reduce the distance between any two vertices in $P_4$ is that $G$ has a vertex adjacent to $v_1$ and $v_4$.
Assume $u\in V(G)$ such that $u$ is adjacent at least with $v_1$ and $v_4$.  Then $D(G,X_G)$ has the following submatrix



\[
M=D(G,X_G)[V(P_4)\cup \{u\},V(P_4)\cup \{u\}]=
\begin{bmatrix}

    x_1 & 1 & 2 & 2 & 1\\
    1 & x_2 & 1 & 2 & a\\
    2 & 1 & x_3 & 1 & b\\
    2 & 2 & 1 & x_4 & 1\\
    1 & a & b & 1 & u

\end{bmatrix}
\]

Note that since $\det(M[\{v_2,v_4\},\{v_1,v_3\}])=-1$, we have that $I^\mathcal{R}_2(G,X_{G})$ is trivial.
Therefore, $P_4$ and any graph containing $P_4$ as an induced subgraph have trivial second distance ideal.
\end{example}

\subsection{Distance ideals of complete graphs and star graphs}

Another interpretation of the distance matrix is the following.
Given a connected graph $G$, the \textit{complete multigraph} $\mathcal{K}(G)$ is a multigraph whose underlying simple graph is a complete graph with $V(G)$ as vertex set, and the number of edges between two vertices $u$ and $v$ is $d_G(u,v)$.
Note that the distance matrix of $G$ is equal to the adjacency matrix of the complete multigraph $\mathcal{K}(G)$.
The converse is not always possible.  That is, for an arbitrary complete multigraph, it is not always possible to find a graph whose distance matrix is equal to the adjacency matrix of this complete multigraph.

The torsion part of the cokernel of the adjacency matrix of a graph $G$ is known as the {\it Smith group} of $G$ and is denoted $S(G)$; see \cite{rushanan}.
Another interesting group is the {\it critical group} which is computed through the Smith normal form of the Laplacian matrix of $G$; see \cite{merino}.
In this way, by computing the Smith normal form of the distance matrix of a graph $G$, we are also computing the Smith group of $\mathcal{K}(G)$.
Furthermore, the critical ideals of a complete multigraph $\mathcal{K}(G)$ coincide with the distance ideals of $G$ evaluated at $-x_v$ 
for all $v\in V(G)$.
The {\it generalized Laplacian matrix} $L(G,X_G)$ of $G$ is the matrix $\diag(X_G)-A(G)$, where $A(G)$ is the adjacency matrix of $G$.
The {\it critical ideals} of $G$ are the ideals $\langle minors_i(L(G,X_G)) \rangle$ for $i\in [n]$.
These ideals were defined in \cite{CV} and further studied in \cite{alfacorrval,AVV,alfalin,AV,AV1}, from which our study was originally inspired.

We finish this section by giving a description of the distance ideals of the complete graphs and the star graphs.
In what follows $\mathcal{R}$ will be a commutative ring containing the integers.

The only case when $G$ and $\mathcal{K}(G)$ are the same is when $G$ is the complete graph.
Therefore for this case, distance ideals and critical ideals are similar.
Since the description of the distance ideals of the complete graph will be used later, we give this description.

\begin{remark}
    In the following, we are going to consider $\prod_{\emptyset}=1$.
\end{remark}

\begin{theorem}\cite[Proposition 3.15 and Theorem 3.16]{CV}\label{distanceidealscompletegraphs}
    The $i$-th distance ideal of the complete graph $K_n$ with $n$ vertices is generated by
     \[
    \begin{cases}
        \prod_{j=1}^n(x_j-1) + \sum_{k=1}^n\prod_{j\neq k}(x_j-1) & \text{if } i = n,\\
         \left\{\prod_{j\in \mathcal{I}}(x_j-1) : \mathcal{I} \subset [n] \text{ and } |\mathcal{I}|=i-1\right\} & \text{if } i < n.
    \end{cases}
    \]
\end{theorem}

Following Proposition~\ref{teo:eval1}, by evaluating the distance ideals at $x_v=0$ for each $v\in V$, we obtain the Smith normal form of distance matrix over the integers of the complete graph.

\begin{corollary}
    The Smith normal form of the distance matrix of the complete graph with $n$ vertices is ${\sf I}_{n-1}\oplus (n-1)$.
\end{corollary}
\begin{proof}
    After the evaluation, we have $\Delta_i=1$, for $i\in [n-1]$.
    And $\Delta_n=|(-1)^n+n(-1)^{n-1}|=n-1$.
\end{proof}

Furthermore, the Smith normal form of other variants of the distance matrix can be computed from Theorem~\ref{distanceidealscompletegraphs}.
Let $tr(u)$ denote {\it transmission} of a vertex $u$, which is defined as $\sum_{v\in V}d_G(u,v)$.
The distance Laplacian matrix is defined as $-D(G,X_G)|_{x_u=-tr(u)}$.
Thus by evaluating the distance ideals at $x_v=-n+1$ for each $v\in V$ we can obtain the Smith normal form of distance Laplacian matrix of the complete graph.
As explained before, this case coincides with the invariant factors of the critical group of the complete graph.

\begin{corollary}
    The Smith normal form of the distance Laplacian matrix of the complete graph with $n$ vertices is $1\oplus n{\sf I}_{n-2}\oplus 0$.
\end{corollary}

Now, let us compute the distance ideals of the star graphs.
For this, we first give a more general result than Theorem~\ref{distanceidealscompletegraphs}.

\begin{proposition}\label{propo:detcompletemultipartite}
    Let $M_n(m)=\diag(X_n)-m{\sf I}_n+m{\sf J}_n$.
    Then
    \[
        \det(M_n(m))=\prod_{i=1}^n(x_i-m)+m\sum_{i=1}^n\prod_{j\neq i}(x_j-m)
    \]
\end{proposition}
\begin{proof}
    For $n=2$, the result follows since $(x_1-m)(x-2-m)+m(x_1-m+x_2-m)=x_1x_2-m^2$.
    Assume $n\geq 2$.
    For simplicity, $M_n$ will denote $M_n(m)$.
    \begin{eqnarray*}
        \det(M_{n+1}) & = & x_{n+1}\cdot\det(M_n) + \sum_{i=1}^n(-1)^{(n+1)+i+(n-i)}\cdot m\cdot\det(M_n|_{x_i=m}) \\
         & = & x_{n+1}\left( \prod_{i=1}^n(x_i-m)+m\sum_{i=1}^n\prod_{j\neq i}(x_j-m)\right)\\
         & & -m\sum_{i=1}^n\left.\left( \prod_{k=1}^n(x_k-m)+m\sum_{k=1}^n\prod_{j\neq k}(x_j-m)\right)\right|_{x_i=m}\\
         & = & x_{n+1} \prod_{i=1}^n(x_i-m) + m\cdot(x_{n+1}-m)\sum_{i=1}^n\prod_{j\neq i}(x_j-m)\\
         & = & \prod_{i=1}^{n+1}(x_i-m)+m\sum_{i=1}^{n+1}\prod_{j\neq i}(x_j-m)
    \end{eqnarray*}
\end{proof}

Thus we have the following result.

\begin{proposition}\label{prop:detgencompletegraph}
    Let $M_n(m)=\diag(X_n)-m{\sf I}_n+m{\sf J}_n$.
    Let
    \[
    A_k=\left\{ m\prod_{i\in \mathcal{I}}(x_i-m) : \mathcal{I}\subset [n] \text{ and } |\mathcal{I}| = k-1\right\}
    \]
    and
    \[
    B_k=\left\{ \prod_{i\in \mathcal{I}}(x_i-m)-m\sum_{i\in \mathcal{I}}\prod_{j\in \mathcal{I}\setminus{i}}(x_j-m) : \mathcal{I}\subset [n] \text{ and } |\mathcal{I}| = k  \right\}.
    \]
    Then, for $k\in [n-1]$, $\left\langle {\rm minors}_k(M_n(m)) \right\rangle$ is equal to $\langle A_k\cup B_k\rangle$.
\end{proposition}
\begin{proof}
    Let $M$ be a $k\times k$ submatrix of $M_n(m)$.
    Then, there exist subsets $\mathcal{I}$ and $\mathcal{J}$ of $[n]$ with $\mathcal{J}\subseteq \mathcal{I}$ and $|\mathcal{I}|=k$ such that $M$ is equivalent to $M_n(m)[\mathcal{I}]|_{\{x_j=m \text{ for all } j \in \mathcal{J}\}}$.
    If $|\mathcal{J}|\geq 2$, then $\det(M)=0$.
    If $|\mathcal{J}|=1$, then, by Proposition~\ref{propo:detcompletemultipartite}, $\det(M)=\pm m\prod_{i\in \mathcal{I}\setminus \mathcal{J}}(x_i-m)$.
    And if $|\mathcal{J}|=0$, then, by Proposition~\ref{propo:detcompletemultipartite}, $\det(M)=\prod_{i\in \mathcal{I}}(x_i-m)-m\sum_{i\in \mathcal{I}}\prod_{j\in \mathcal{I}\setminus{i}}(x_j-m)$.
    Thus, we have one containment.
    The other one follows since by an appropriate selection of the indices $\mathcal{I}$ and $\mathcal{J}$, we can obtain any element in $A_k\cup B_k$ as the determinant of $M_n(m)[\mathcal{I},\mathcal{J}]$.
\end{proof}

\begin{theorem}\label{teo:detK1m}
    Let $m\geq 1$ and
    \[
        M(m)=
        \begin{bmatrix}
        \diag(X_m)-2{\sf I}_m+2{\sf J}_m & {\sf J}_{m,1}\\
        {\sf J}_{1,m} & y\\
        \end{bmatrix}.
    \]
    Then, for $i\in [m]$, $\det(M(m)[m+1,i])$ is equal to
    \begin{equation}
        (-1)^{m-i}\prod\limits^m_{\substack{j=1\\ j\neq i}}(x_j-2)
    \end{equation}
    And, $\det(M(m))$ is equal to
    \begin{equation}
        y\prod\limits_{i=1}^m(x_i-2) + (2y-1)\sum\limits_{i=1}^m\prod\limits^m_{\substack{j=1\\ j\neq i}}(x_j-2)
    \end{equation}
\end{theorem}

\begin{proof}
    For simplicity, let $N$ denote $\diag(X_m)-2{\sf I}_m+2{\sf J}_m$.
    Since $(M(m)[m+1,i])[j,m]$ is equivalent to $N[j]|_{x_i=2}$, up to $|i-j|-1$ column switchings when $|i-j|\geq 2$, then
\[
    \det(M(m)[m+1,i])=\sum_{j=1}^m(-1)^{m+j}\theta(i,j)\det N[j]|_{x_i=2},
\]
where
\[
\theta(i,j)=
\begin{cases}
    (-1)^{|i-j|-1} & \text{ if }|i-j|\geq 2\\
    1 & \text{otherwise.}
\end{cases}
\]
From which follows that $\det(M(m)[m+1,i])$ is equal to
\[
    (-1)^{m-i-1}\sum_{j=1}^m\delta(i,j)\det N[j]|_{x_i=2},
\]
where
\[
\delta(i,j)=
\begin{cases}
    -1 & \text{ if }i=j\\
    1 & \text{otherwise.}
\end{cases}
\]
By Proposition~\ref{propo:detcompletemultipartite}, $\det(M(m)[m+1,i])$ is equal to
\[
    (-1)^{m-i-1}\sum_{j=1}^m\delta(i,j) \left[ \prod_{\substack{k=1\\ k\neq j}}^m(x_k-2)+2\sum_{\substack{k=1\\ k\neq j}}^m\prod_{\substack{l\neq k\\ l\neq j}}(x_l-2) \right]_{x_i=2},
\]
which is also equal to
\[
    (-1)^{m-i-1}\left(-\prod_{\substack{k=1\\ k\neq i}}^m(x_k-2) + 2\sum_{j=1}^m\delta(i,j) \left[ \sum_{\substack{k=1\\ k\neq j}}^m\prod_{\substack{l\neq k\\ l\neq j}}(x_l-2) \right]_{x_i=2}\right).
\]
The result now follows from the following claim.
\begin{claim}
\[
  \sum_{j=1}^m\delta(i,j) \left[ \sum_{\substack{k=1\\ k\neq j}}^m\prod_{\substack{l\neq k\\ l\neq j}}(x_l-2) \right]_{x_i=2}=0.
\]
\end{claim}
\begin{proof}
For a fixed $i$, if $j\neq i$, then $
\left[ \sum_{\substack{k=1\\ k\neq j}}^m\prod_{\substack{l\neq k\\ l\neq j}}(x_l-2) \right]_{x_i=2}
=
\prod_{\substack{l\neq i\\ l\neq j}}(x_l-2)$.
From this, it follows that
\[
  \sum_{j=1}^m\delta(i,j) \left[ \sum_{\substack{k=1\\ k\neq j}}^m\prod_{\substack{l\neq k\\ l\neq j}}(x_l-2) \right]_{x_i=2}
  =
  -\sum_{\substack{k=1\\ k\neq i}}^m\prod_{\substack{l\neq k\\ l\neq i}}(x_l-2)
  +
  \sum_{\substack{j=1\\ j\neq i}}^m\prod_{\substack{l\neq i\\ l\neq j}}(x_l-2)
  =
  0.
\]
\end{proof}

Finally,
\begin{eqnarray*}
    \det M(m) & = & y\det N + \sum_{i=1}^m(-1)^{m+1+i}\det(M(m)[m+1,i])\\
    & = & y\left( \prod_{i=1}^m(x_i-2)+2\sum_{i=1}^m\prod_{j\neq i}(x_j-2) \right) - \sum_{i=1}^m\prod\limits^m_{\substack{j=1\\ j\neq i}}(x_j-2)\\
    & = & y\prod\limits_{i=1}^m(x_i-2)+(2y-1)\sum\limits_{i=1}^m\prod\limits^m_{\substack{j=1\\ j\neq i}}(x_j-2).
\end{eqnarray*}
\end{proof}

\begin{theorem}
    Let
    \[
    C_k=\left\{ \prod\limits_{i\in \mathcal{I}}(x_i-2) : \mathcal{I}\subset [m] \text{ and } |\mathcal{I}| = k-1\right\}
    \]
    and
    \[
    D_k=\left\{ (2y-1)\prod\limits_{i\in \mathcal{I}}(x_i-2) : \mathcal{I}\subset [m] \text{ and } |\mathcal{I}| = k-2  \right\}.
    \]
    For $k \in [n-1]$, the $k$-th distance ideal of the star graph $K_{m,1}$ is generated by $\langle C_k \cup D_k\rangle$.

\end{theorem}
\begin{proof}
    Let $M(m)$ be a matrix defined as in Theorem~\ref{teo:detK1m}, and $N=M(m)[\mathcal{I},\mathcal{J}]$ with $|\mathcal{I}|=|\mathcal{J}|=k$.
    There are 3 possible cases:
    \begin{enumerate}[a.]
        \item Both sets $\mathcal{I}$ and $\mathcal{J}$ contain $m+1$,
        \item Only one of the sets $\mathcal{I}$ or $\mathcal{J}$ contains $m+1$, and
        \item Neither $\mathcal{I}$ nor $\mathcal{J}$ contains $m+1$.
    \end{enumerate}
In case (a), $N$ is equivalent to $M(m)[\mathcal{I}]|_{\{x_i = 2 \,:\, i \in \mathcal{I}\setminus\mathcal{J} \}}$.
Note that if $|\mathcal{I}\setminus\mathcal{J}|\geq 2$, then $\det(N)=0$.
If $|\mathcal{I}\setminus\mathcal{J}|=1$, then by adequately applying Equation 2 of Theorem~\ref{teo:detK1m} we obtain
\begin{eqnarray*}
\det(N) & = & \left[y\prod_{i\in \mathcal{I}\setminus\{ m+1\}}(x_i-2)+(2y-1)\sum_{i\in \mathcal{I}\setminus \{m+1\}}\prod_{j\in\mathcal{I}\setminus \{i,m+1\}}(x_j-2)\right]_{\{x_i = 2 \,:\, i \in \mathcal{I}\setminus\mathcal{J} \}}\\
 & = & (2y-1)\prod_{ j\in\mathcal{I}\cap\mathcal{J}\setminus\{ m+1 \} }(x_j-2) \in D_k.
\end{eqnarray*}
If $\mathcal{I}=\mathcal{J}$, then
by applying Equation 2 of Theorem~\ref{teo:detK1m} we obtain
\begin{eqnarray*}
\det(N) & = & y\prod_{i\in \mathcal{I}\setminus\{ m+1\}}(x_i-2)+(2y-1)\sum_{i\in \mathcal{I}\setminus \{m+1\}}\prod_{j\in\mathcal{I}\setminus \{i,m+1\}}(x_j-2),
\end{eqnarray*}
which is in $\langle C_k \cup D_k\rangle$.

In case (b), let us assume, without loss of generality, $m+1\in \mathcal{I}$.
We have $N$ is equivalent to $M(m)[\mathcal{I}]|_{\{x_i = 2 \,:\, i \in \mathcal{J} \setminus \mathcal{I} \}}$.
Note that if $|\mathcal{J} \setminus \mathcal{I}|\geq 2$, then $\det(N)=0$, and $|\mathcal{J} \setminus \mathcal{I}|\neq 0$ since otherwise $m+1$ would be in $\mathcal{J}$.
Thus $|\mathcal{J} \setminus \mathcal{I}|=1$, then
by applying Equation 1 of Proposition~\ref{teo:detK1m} we obtain $\det(N)$ is equal to, up to sign, $\prod_{i\in\mathcal{J} \cap \mathcal{I}}(x_i-2)$ which is in $C_k$.

Finally, in case (c), we have that $\det(N)$ is in $A_k$ or $B_k$ of Proposition~\ref{prop:detgencompletegraph}.
The result now follows since $\langle A_k\cup B_k\rangle\subset \langle C_k\rangle$.

The other statement can be derived from cases (a) and (b).
\end{proof}

As in the case of complete graphs, this description could be used to give the Smith normal form of the distance matrix and distance Laplacian matrix of the star graphs over the integers.

\begin{corollary}
    The Smith normal form of the distance matrix of the star graph with $m$ leaves is ${\sf I}_{2}\oplus 2 {\sf I}_{n-2}\oplus 2m$.
\end{corollary}
\begin{proof}
    After the evaluating the distance ideal at $X_G={\bf 0}$, we have $\Delta_i=1$, for $i\in [2]$; $\Delta_i=2^{(i-2)}$, for $i\in \{3,...,m\}$; and $\Delta_{m+1}=2^{m-1}m$.
\end{proof}

\begin{corollary}
    The Smith normal form of the distance Laplacian matrix of the complete graph with $n$ vertices is ${\sf I}_{m}\oplus 2m(m-1)$.
\end{corollary}
\begin{proof}
    After evaluating the distance ideal at $x_i=1$ for $i\in[m]$ and $y=m$, we obtain $\Delta_i=1$ for $i\in[m]$, and $\Delta_{m+1}=2m(m-1)$, from which the invariant factors can be easily obtained.
\end{proof}

\section{Graphs with at most one trivial distance ideal}\label{section:classification}
Despite the fact that distance ideals are, in general, not monotone under taking induced subgraphs, we will be able to classify the graphs which have exactly 1 trivial distance ideal over $\mathbb{Z}$ and $\mathbb{R}$ in terms of forbidden induced subgraphs.

Let $\Lambda^{\mathcal{R}}_k$ denote the family of graphs with at most $k$ trivial distance ideals over $\mathcal{R}$.
A graph $G$ is {\it forbidden} for $\Lambda^{\mathcal{R}}_k$ if the $(k+1)$-th distance ideal, over ${\mathcal{R}}$, of $G$ is trivial.
The set of forbidden graphs for $\Lambda^{\mathcal{R}}_k$ will be denoted by ${\sf Forb}^{\mathcal{R}}_k$.
In addition, a graph $G\in {\sf Forb}^{\mathcal{R}}_k$ is {\it minimal} if $G$ does not contain a graph in ${\sf Forb}^{\mathcal{R}}_k$ as induced subgraph, and for any graph $H$ containing $G$ as induced subgraph, $H\in{\sf Forb}^{\mathcal{R}}_k$.


First we consider the case over $\mathbb{Z}$.

\begin{figure}[h!]
\begin{center}
\begin{tabular}{c@{\extracolsep{10mm}}c@{\extracolsep{10mm}}c@{\extracolsep{10mm}}c@{\extracolsep{10mm}}c}
	\begin{tikzpicture}[scale=.7]
	\tikzstyle{every node}=[minimum width=0pt, inner sep=2pt, circle]
	\draw (126+36:1) node (v1) [draw] {};
	\draw (198+36:1) node (v2) [draw] {};
	\draw (270+36:1) node (v3) [draw] {};
	\draw (342+36:1) node (v4) [draw] {};
	\draw (v1) -- (v2);
	\draw (v2) -- (v3);
	\draw (v4) -- (v3);
	\end{tikzpicture}
&
	\begin{tikzpicture}[scale=.7]
	\tikzstyle{every node}=[minimum width=0pt, inner sep=2pt, circle]
	\draw (-.5,-.9) node (v1) [draw] {};
	\draw (.5,-.9) node (v2) [draw] {};
	\draw (0,0) node (v3) [draw] {};
	\draw (0,.9) node (v4) [draw] {};
	\draw (v1) -- (v2);
	\draw (v1) -- (v3);
	\draw (v2) -- (v3);
	\draw (v3) -- (v4);
	\end{tikzpicture}
&
	\begin{tikzpicture}[scale=.7]
	\tikzstyle{every node}=[minimum width=0pt, inner sep=2pt, circle]
	\draw (-.5,0) node (v2) [draw] {};
	\draw (0,-.9) node (v1) [draw] {};
	\draw (.5,0) node (v3) [draw] {};
	\draw (0,.9) node (v4) [draw] {};
	\draw (v1) -- (v2);
	\draw (v1) -- (v3);
	\draw (v2) -- (v3);
	\draw (v2) -- (v4);
	\draw (v3) -- (v4);
	\end{tikzpicture}
\\
$P_4$
&
$\sf{paw}$
&
$\sf{diamond}$
\end{tabular}
\end{center}
\caption{The graphs $P_4$, $\sf{paw}$ and $\sf{diamond}$.}
\label{fig:forbiddendistance1}
\end{figure}
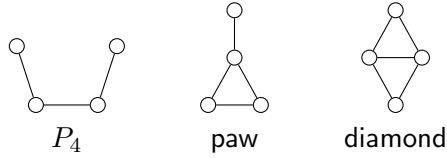

\begin{lemma}\label{lemma:P4PawDiamondForbidden}
    The graphs $P_4$, $\sf{paw}$ and $\sf{diamond}$ are minimal forbidden graphs for graphs with 1 trivial distance ideal over $\mathbb{Z}$.
\end{lemma}
\begin{proof}
    The fact that these are forbidden graphs follows from the observation that $P_4$, $\sf{paw}$ and $\sf{diamond}$ have exactly 2 trivial distance ideals over $\mathbb{Z}$, this can be verified with the code in the Appendix.
The minimality follows from Lemma \ref{lemma:distance2inducedmonotone} and Example \ref{example:P4inducemonotone}, and the fact that no proper induced subgraph of these graphs has 2 trivial distance ideals over $\mathbb{Z}$.
\end{proof}

\begin{figure}[h!]
\begin{center}
\begin{tabular}{c@{\extracolsep{10mm}}c@{\extracolsep{10mm}}c@{\extracolsep{10mm}}c@{\extracolsep{10mm}}c}
   \begin{tikzpicture}[scale=.7]
	\tikzstyle{every node}=[minimum width=0pt, inner sep=2pt, circle]
	\draw (126-36:1) node (v1) [draw] {};
	\draw (198-36:1) node (v2) [draw] {};
	\draw (270-36:1) node (v3) [draw] {};
	\draw (342-36:1) node (v4) [draw] {};
	\draw (414-36:1) node (v5) [draw] {};
	\draw (v1) -- (v2);
	\draw (v1) -- (v5);
	\draw (v2) -- (v3);
	\draw (v2) -- (v4);
	\draw (v2) -- (v5);
	\draw (v3) -- (v4);
	\draw (v3) -- (v5);
	\draw (v4) -- (v5);
	\end{tikzpicture}
&
	\begin{tikzpicture}[scale=.7]
	\tikzstyle{every node}=[minimum width=0pt, inner sep=2pt, circle]
	\draw (0:1) node (v1) [draw] {};
	\draw (60:1) node (v2) [draw] {};
	\draw (120:1) node (v3) [draw] {};
	\draw (180:1) node (v4) [draw] {};
	\draw (240:1) node (v5) [draw] {};
	\draw (300:1) node (v6) [draw] {};
	\draw (v1) -- (v2);
	\draw (v1) -- (v3);
	\draw (v1) -- (v5);
	\draw (v1) -- (v6);
	\draw (v2) -- (v4);
	\draw (v2) -- (v5);
	\draw (v2) -- (v6);
	\draw (v3) -- (v4);
	\draw (v3) -- (v5);
	\draw (v3) -- (v6);
	\draw (v4) -- (v5);
	\draw (v4) -- (v6);
	\draw (v5) -- (v6);
	\end{tikzpicture}
&
	\begin{tikzpicture}[scale=.7]
	\tikzstyle{every node}=[minimum width=0pt, inner sep=2pt, circle]
	\draw (-.5,-.9) node (v1) [draw] {};
	\draw (.5,-.9) node (v2) [draw] {};
	\draw (0,0) node (v3) [draw] {};
	\draw (-.5,.9) node (v4) [draw] {};
	\draw (.5,.9) node (v5) [draw] {};
	\draw (v1) -- (v2);
	\draw (v1) -- (v3);
	\draw (v2) -- (v3);
	\draw (v3) -- (v4);
	\draw (v3) -- (v5);
	\end{tikzpicture}
&
	\begin{tikzpicture}[scale=.7]
	\tikzstyle{every node}=[minimum width=0pt, inner sep=2pt, circle]
	\draw (-.5,0) node (v2) [draw] {};
	\draw (0,-.9) node (v1) [draw] {};
	\draw (.5,0) node (v3) [draw] {};
	\draw (1.5,0) node (v5) [draw] {};
	\draw (0,.9) node (v4) [draw] {};
	\draw (v1) -- (v2);
	\draw (v1) -- (v3);
	\draw (v2) -- (v3);
	\draw (v2) -- (v4);
	\draw (v3) -- (v4);
	\draw (v3) -- (v5);
	\end{tikzpicture}
\\
$K_5\setminus P_2$
&
$K_6\setminus M_2$
&
$\ltimes$
&
\sf{dart}
\end{tabular}
\end{center}
\caption{The graphs $K_5\setminus P_2$, $K_6\setminus M_2$, $\ltimes$ and $\sf dart$.}
\label{fig:forbiddencritical2}
\end{figure}
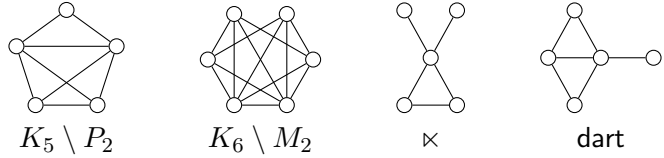

Given a family $\mathcal{F}$ of graphs, a graph is $\mathcal{F}$-free if  no induced subgraph of $G$ is isomorphic to a graph in $\mathcal{F}$.

\begin{lemma}\cite[Theorem 4.3]{AV}\label{lem:classificationgamma2}
   A simple connected graph is $\{P_4, K_5\setminus P_2, K_6\setminus M_2, \ltimes, \sf{dart}\}$-free if and only if it is an induced subgraph of $K_{m,n,o}$ or $\overline{K_n} \vee (K_m+K_o)$.
\end{lemma}

\begin{proposition}\label{prop:P4PawDiamondIsInKKL}
   If a simple connected graph is $\{P_4, \sf{paw}, \sf{diamond}\}$-free, then it is an induced subgraph of $K_{m,n,o}$ or $\overline{K_n} \vee (K_m+K_o)$.
\end{proposition}
\begin{proof}
First note that $\sf{paw}$ is an induced subgraph of $K_5\setminus P_2$, $\ltimes$ and $\sf dart$, and $\sf{diamond}$ is an induced subgraph of $K_6\setminus M_2$.
Therefore, if $G$ is $\{P_4, \sf{paw}, \sf{diamond}\}$-free, then $G$ is $\{P_4, K_5\setminus P_2, K_6\setminus M_2, \ltimes, \sf{dart}\}$-free.
The result then follows by Lemma \ref{lem:classificationgamma2}.
\end{proof}

Now, we have the following characterization.

\begin{theorem}\label{teo:classification}
For $G$ a simple connected graph, the following are equivalent:
\begin{enumerate}
\item $G$ has only 1 trivial distance ideal over $\mathbb{Z}$.
\item $G$ is $\{P_4,\sf{paw},\sf{diamond}\}$-free.
\item $G$ is an induced subgraph of $K_{m,n}$ or $K_{n}$.
\end{enumerate}
\end{theorem}

\begin{proof}
$(1)\implies (2)$ follows from Lemma \ref{lemma:P4PawDiamondForbidden}.

$(2)\implies (3)$: 
By Proposition \ref{prop:P4PawDiamondIsInKKL}, $G$ is an induced subgraph of $K_{m,n,o}$ or $\overline{K_n} \vee (K_m+K_o)$.
However, there are induced subgraphs in $K_{m,n,o}$ and $\overline{K_n} \vee (K_m+K_o)$ isomorphic to $\sf{paw}$ or $\sf{diamond}$.
By inspection, we are going to determine that $G$ is an induced subgraph of $K_{m,n}$ or $K_{n}$.

If $m,n\geq 1$ and $o\geq 2$, then $K_{m,n,o}$ contains $\sf{diamond}$ as induced subgraph.
Therefore, $o\leq 1$.
For simplicity, we assume $m\geq n\geq o$.
Thus, we have two cases:
\begin{enumerate}
   \item $o=0$, or
   \item $o=1$.
\end{enumerate}
In the first case, $G=K_{m,n}$.
In the second case, $K_{1,1,1}$ is the unique possibility, because if $m\geq 2$ and $n\geq 1$, then $K_{m,n,1}$ would contain $\sf{diamond}$ as induced subgraph.
Indeed, $K_{2,1,1}$ is isomorphic to $\sf{diamond}$.

If $m\geq 2$ and $n\geq 2$, then $\overline{K_n} \vee (K_m+K_o)$ contains $\sf{diamond}$ as induced subgraph.
For simplicity, we assume $m\geq o$.
Thus, we have two cases:
\begin{enumerate}
   \item $m\leq 1$, or
   \item $n\leq 1$.
\end{enumerate}
For case 1, we have that $o\leq m\leq 1$ and $n\geq 2$, thus $G$ is isomorphic to a bipartite graph $K_{2,n}$.
And for case 2, we have two cases: either $m\geq 2$ or $m=1$.
In the first case, if $n=1$, then $o=0$, otherwise $\sf{paw}$ will be an induced subgraph of $\overline{K_1} \vee (K_m+K_o)$.
But $\overline{K_1} \vee (K_m)$ is isomorphic to a complete graph with $m+1$ vertices.
In the second case, $G$ is an induced subgraph of $\overline{K_1} \vee (K_1+K_1)\cong P_3$.

$(3)\implies (1)$: Note that any non-trivial connected graph has trivial first distance ideal.
For an isolated vertex we have $I^\mathbb{Z}_1(K_1,\{ x\}) =\left< x \right>$.
Now we have to compute the second distance ideals of $K_{n}$ and $K_{m,n}$.
The 2-minors of the generalized distance matrix of a complete graphs are of the forms $x_ix_j-1$ and $x_i-1$.
Since $x_ix_j-1\in \left< x_1-1, \dots, x_n-1 \right>$,
\begin{equation}\label{eqn:gamma1}
	I^\mathbb{Z}_2(K_n,X_{K_n})=
	\begin{cases}
		\left< x_1x_2-1 \right> & \text{if } n=2, \text{ and,}\\
		\left< x_1-1, \dots, x_n-1 \right> & \text{if } n\geq 3.
	\end{cases}
\end{equation}
Thus complete graphs have at most one trivial distance ideal.

If $m\geq2$ and $n=1$, then the 2-minors of $D(K_{m,1},\{x_1, \dots, x_m, y\})$ of $K_{m,1}$ have one of the following forms: $x_ix_j-4$, $2x_i-4$, $x_i-2$, $x_iy-1$ and $2y-1$.
Thus
\[
I^\mathbb{Z}_2(K_{m,1},\{x_1, \dots, x_m, y\})=\langle x_1-2, \dots, x_m-2,2y-1\rangle.
\]
If $m\geq2$ and $n\geq2$, then the 2-minors of $D(K_{m,n},\{x_1, \dots, x_m, y_1, \dots, y_n\})$ of $K_{m,n}$ have one of the following forms: $x_ix_j-4$, $2x_i-4$, $x_i-2$, $x_iy_j-1$, $2x_i-1$,
$y_iy_j-4$, $2y_i-4$, $y_i-2$, $2y_i-1$ and 3.
Thus
\[
I^\mathbb{Z}_2(K_{m,n},\{x_1, \dots, x_m, y_1, \dots, y_n\})=\langle x_1-2, \dots, x_m-2, y_1-2, \dots, y_n-2,3\rangle.
\]
Therefore bipartite graphs have at most one trivial distance ideal.
\end{proof}

We finish this section by classifying graphs which have exactly 1 trivial distance ideal over $\mathbb{R}$.


\begin{lemma}\label{lemma:P4PawDiamondC4Forbidden}
    The graphs $P_4$, $\sf{paw}$, $\sf{diamond}$ and $C_4$ are minimal forbidden graphs for graphs with 1 trivial distance ideal over $\mathbb{R}$.
\end{lemma}
\begin{proof}
    The graphs $P_4$, $\sf{paw}$, $\sf{diamond}$ and $C_4$ have exactly 2 trivial distance ideals over $\mathbb{R}$, which can be verified with the code in the Appendix.
The minimality of $\sf{paw}$, $\sf{diamond}$ and $C_4$ follows from Lemma \ref{lemma:distance2inducedmonotone}.  Minimality of $P_4$ follows from Example \ref{example:P4inducemonotone} and the fact that no proper induced subgraph of these graphs has 2 trivial distance ideals over $\mathbb{R}$.
\end{proof}

\begin{theorem}\label{teo:classification2}
For $G$ a simple connected graph, the following are equivalent:
\begin{enumerate}
\item $G$ has only 1 trivial distance ideal over $\mathbb{R}$.
\item $G$ is $\{P_4,\sf{paw},\sf{diamond}, C_4\}$-free.
\item $G$ is an induced subgraph of $K_{1,n}$ or $K_{n}$.
\end{enumerate}
\end{theorem}
\begin{proof}
The statement can be derived from Lemma \ref{lemma:P4PawDiamondC4Forbidden}, Theorem~\ref{teo:classification} and the observation that 
$I^\mathbb{R}_2\left(K_{m,n},X_{K_{m,n}}\right)$ is trivial when $m\geq n\geq 2$.
\end{proof}

A graph is {\it trivially perfect} if for every induced subgraph the stability number equals the number of maximal cliques.
In \cite[Theorem 2]{G}, Golumbic characterized trivially perfect graphs as $\{P_4,C_4\}$-free graphs. There are other equivalent characterization of this family; see \cite{B,CCY,Rubio}.
Therefore, from Theorem~\ref{teo:classification2}, graphs with 1 trivial distance ideal over $\mathbb{R}$ are a subclass of trivially perfect graphs.

A related family of graphs come from the {\it graph sandwich problem} for property $\Pi$, which is defined as follows.
Given two graphs $G_1 = (V, E_1)$ and $G_2 = (V, E_2)$ such that $E_1 \subseteq E_2$, is there a graph $G = (V, E)$ such that $E_1 \subseteq E \subseteq E_2$ which satisfies property $\Pi$?
In the literature there are several characterizations where the problem restricted to the graphs found in Theorem~\ref{teo:classification2} lies certain complexity class.
For instance, in \cite{DFMT} that the {\sf paw}-free graph sandwich problem is in ${\sf P}$.
See also \cite{G1}.

In \cite[Theorem 3]{HW} it was proved that the distance matrices of trees have exactly 2 invariant factors equal to 1.
This differs from the critical group, since the Laplacian matrix of any tree has all invariant factors equal to 1.
An interesting and difficult question will be to characterize the graphs whose distance matrix has at most 2 invariant factors equal to 1.






\section{Acknowledgements}
C.A. Alfaro was partially supported by SNI and CONACyT.

\appendix
\section{Computing distance ideals with Macaulay 2 on SageMath}

We give a code for computing the distance ideals of graphs over $\mathbb{Z}$ with Macaulay2 (see \cite{macaulay2}) using the widely used interface of SageMath (see \cite{sage}).

\begin{lstlisting}
# The input g is a graph
def DistanceIdealsZZ(g):
    n = g.order()
    Distance = matrix(n)
    for i in range(n):
        for j in range(i,n):
            Distance[j,i] = g.distance(j,i)
            Distance[i,j] = Distance[j,i]
    S='['
    for i in range(n):
        if i > 0 :
            S = S + ","
        S = S + 'x' + str(i)
    S=S+']'
    R = macaulay2.ring("ZZ",S).to_sage()
    R.inject_variables()
    GDistance = diagonal_matrix(list(R.gens())) + Distance
    print(GDistance)
    for i in range(n+1):
        I = R.ideal(GDistance.minors(i))
        print("Distance ideals of size " + str(i))
        print(I.groebner_basis())
\end{lstlisting}
Thus for computing the distance ideals over $\mathbb{Z}$ of the cycle with 4 vertices is the following.

\begin{lstlisting}[firstnumber=23]
C4 = graphs.CycleGraph(4)
DistanceIdealsZZ(C4)
\end{lstlisting}
The output is the following.

\begin{lstlisting}[numbers=none]
Defining x0, x1, x2, x3
[x0  1  2  1]
[ 1 x1  1  2]
[ 2  1 x2  1]
[ 1  2  1 x3]
Distance ideals of size 0
[1]
Distance ideals of size 1
[1]
Distance ideals of size 2
[x0 + 1, x1 + 1, x2 + 1, x3 + 1, 3]
Distance ideals of size 3
[x0*x1 - 2*x0 - 2*x1 + 4, 2*x0*x2 - x0 - x2 - 4, x0*x3 - 2*x0 - 2*x3 + 4, x1*x2 - 2*x1 - 2*x2 + 4, 2*x1*x3 - x1 - x3 - 4, x2*x3 - 2*x2 - 2*x3 + 4]
Distance ideals of size 4
[x0*x1*x2*x3 - x0*x1 - 4*x0*x2 - x0*x3 + 4*x0 - x1*x2 - 4*x1*x3 + 4*x1 - x2*x3 + 4*x2 + 4*x3]
\end{lstlisting}

To compute the ideals over other principal ideal domains, it suffices to replace line 15.
For instance, to compute the ideals over $\mathbb{Q}$, line 15 should be changed to the following.
\begin{lstlisting}[numbers=none]
    R = macaulay2.ring("QQ",S).to_sage()
\end{lstlisting}


\end{document}